\documentclass[11pt,reqno]{amsart}

 \usepackage{mathrsfs}

\newcommand{\tz}{\tilde{z}}

\usepackage{color}
\usepackage{amsmath, amsthm, amssymb}
\usepackage{amsfonts}
\usepackage[dvips]{epsfig}
\usepackage{graphicx}
\usepackage{caption}
\usepackage{subcaption}
\usepackage[english]{babel}
\usepackage{hyperref}

\usepackage{tikz}
\usepackage{rotating}
\usepackage[utf8]{inputenc}
\usepackage{cite}
\usepackage{amscd}
\usepackage{color}
\usepackage{bm}
\usepackage{enumerate}

\usepackage{verbatim}
\usepackage{hyperref}
\usepackage{amstext}
\usepackage{latexsym}
\theoremstyle{plain}
\newtheorem{theorem}{Theorem}[section]

\newtheorem{proposition}[theorem]{Proposition}
\newtheorem{lemma}[theorem]{Lemma}

\numberwithin{theorem}{section}
\numberwithin{equation}{section}

\newcommand{\average}{{\mathchoice {\kern1ex\vcenter{\hrule height.4pt
width 6pt depth0pt} \kern-9.7pt} {\kern1ex\vcenter{\hrule
height.4pt width 4.3pt depth0pt} \kern-7pt} {} {} }}

\def\R{\mathbb{R}}

\renewcommand{\a }{\alpha }
\renewcommand{\b }{\beta }
\renewcommand{\d}{\delta }

\newcommand{\D }{\Delta }

\newcommand{\e }{\varepsilon }

\newcommand{\G }{\Gamma}

\newcommand{\n }{\nabla }
\newcommand{\vp }{\varphi }
\renewcommand{\phi}{\varphi}

\newcommand{\s }{\sigma }

\renewcommand{\th }{\theta }

\renewcommand{\O }{\Omega }

\newcommand{\ov}{\overline}

\newcommand{\be}{\begin{equation}}
\newcommand{\ee}{\end{equation}}

\newcommand{\de}{\partial}

\newcommand{\ti}{\widetilde}

\newcommand{\ra}{{\rangle}}
\newcommand{\la}{{\langle}}

\newcommand{\calC }{\mathcal{C}}

\newcommand{\calD }{\mathcal{D}}

\newcommand{\N}{\mathbb{N}}

\renewcommand{\epsilon}{\varepsilon}

\begin{document}
 
\author{El Hadji Abdoulaye THIAM} 
\title[The Role of the Mean Curvature in a Mixed Hardy-Sobolev Trace Inequality]
{The Role of the Mean Curvature in a Mixed Hardy-Sobolev Trace Inequality}
\maketitle
\begin{abstract}
Let $\O$ be a smooth bounded domain of $\R^{N+1}$  of boundary $\de \O= \G_1 \cup \G_2$ and such that $\de \O \cap \G_2$ is a neighborhood of $0$, $h \in \calC^0(\de \O \cap \G_2) $ and $s \in [0,1)$. We propose to study existence of positive solutions to  the following Hardy-Sobolev trace problem with mixed boundaries conditions
\begin{align}
\begin{cases}
\displaystyle \D u= 0& \qquad \textrm{ in } \O\vspace{2mm}\\
\displaystyle u=0 & \qquad \textrm{ on } \G_1 \vspace{2mm}\\
\displaystyle  \frac{\de u}{\de \nu}=h(x) u + \frac{u^{q(s)-1}}{d(x)^{s}}  & \qquad \textrm{ on } \G_2,
\end{cases}
\end{align}
where $q(s):=\frac{2(N-s)}{N-1}$ is the critical Hardy-Sobolev trace exponent and $\nu$ is the outer unit normal of $\de \O$.
  In particular, we prove existence of minimizers when $N \geq 3$ and the mean curvature is sufficiently below the potential $h$ at $0$.
\end{abstract} 
\bigskip
\noindent
\textit{}
\section{Introduction and main result}\label{Introduction}
For $N \geq 2$ and $s\in [0,1]$, we consider the Hardy-Sobolev trace best constant:
\be\label{eq:CKNtrace} 
S_{N,s}=
\inf_{u\in \calD}\frac{\displaystyle \int_{  \R^{N+1}_+}|\n u|^2 dz }
{\displaystyle \left(\int_{ \de \R^{N+1}_+}|x|^{-s}|u|^{q(s)} d x \right)^{\frac{2}{q(s)}}},
\ee
where $q(s)=\frac{2(N-s)}{N-1}$ is the Hardy-Sobolev trace exponent, see for instance \cite{FMT} and also \cite{CSilv} for generalizations. Here and in the following, we denote by 
$$
\R^{N+1}_+=\left\{z=(z_1,x)\in \R^{N+1}\quad:\quad z_1>0 \right\}
$$
with boundary $\de\R^{N+1}_+=\R^N\times\{0\}\equiv \R^N$.
We denote  and henceforth define $\calD:=\calD^{1,2}(\ov{\R^{N+1}_+})$ the completion of $C^\infty_c(\ov{\R^{N+1}_+} )$  with respect to the norm
$$
u \mapsto \left(\int_{\R^{N+1}_+}|\n u|^2dz\right)^{1/2}.
$$
Note that for $s=0$ then $q(0)=:2^{\sharp}$ is the critical Sobolev trace exponent  while  $S_{N,0}$ coincides with the  Sobolev trace constant studied by Escobar \cite{Esc} and Beckner \cite{Beck} with applications in the Yamabe problem with prescribed mean curvature.
Existence of cylindrical symmetric decreasing minimizers for the quotient  $S_{N,s}$ in \eqref{eq:CKNtrace} were obtained by Lieb [\cite{Lieb}, Theorem 5.1]. If $s=1$, we recover $S_{N,1}=2 \frac {\Gamma^2({\frac{N+1}{4}}) }{\Gamma^2(\frac{N-1}{4})}$, the relativistic Hardy constant (see e.g. \cite{Herbst}) which is never achieved in $\calD$. In this case, it is expected that there is no influence of the curvature in
comparison with the   works on Hardy inequalities with singularity at the boundary or in Riemannian manifolds, see \cite{Fall, Elaj, Elage2, Elage3, Elage4}. \\

We consider a smooth domain $\O$ of $\R^{N+1}$, with boundary $\de \O=\G_1 \cup \G_2$ and such that $\de\O \cap \G_2$ is a smooth neighborhood of $0$. Given $h \in \calC^0\left(\G_2 \cap \de \O\right)$, we suppose the following argument of coercivity: there exists a positive constant $C$, depending on $\O$, such that
\begin{equation}\label{Coercivity}
\int_{\O} |\n u|^2 dz+\int_{\G_2} h(x) u^2 dx \geq C \left(\int_\O |\n u|^2 dz+\int_\O u^2 dz\right) \quad \forall u \in H^1(\O).
\end{equation}
In [\cite{FMT}, Lemma 2.4] authors showed
the existence of  a constant $C_1(\O)>0$ such that
the following inequality holds
$$
C_1(\O){\left(\int_{\de\O}d(x)^{-s} |u|^{q(s)}dx\right)^{\frac{2}{q(s)}}} \leq \int_{\O}|\n u|^2dz
+\int_{\O}| u|^2dz \qquad \forall u\in H^1(\O)
$$
where $d(x):=dist_{\de \O}(0,x)$ is  the Riemannian distance on the boundary $\de \O$ of $\O$.
Then by \eqref{Coercivity} we have the existence of a positive constant $C(\O)$ depending on $\O$ such that
\be\label{Lundi}
C(\O){\left(\int_{\G_2}d(x)^{-s} |u|^{q(s)}dx\right)^{\frac{2}{q(s)}}} \leq \int_{\O}|\n u|^2dx
+\int_{\G_2}h(x) u^2dx \qquad \forall u\in H^1(\O).
\ee

Our aim in this paper is to study the existence of minimizers for the following mixed Hardy-sobolev trace quotient:
\be\label{eq:gOdef-i}
\mu_{s}(\O):= \inf_{u\in H^1(\O)\setminus \lbrace0\rbrace}\frac{  \displaystyle  \int_{\O}|\n u|^2dx
+\int_{\G_2}h(x)| u|^2dx         }{  \displaystyle  \left(\int_{\G_2}d(x)^{-s} |u|^{q(s)}dx\right)^{\frac{2}{q(s)}}},
\ee
for $s\in[0,1)$.
%
Our main  result is the following
\begin{theorem}\label{Theorem1.1}
Let $N \geq 3$, $\O$ be a bounded smooth domain of  $\R^{N+1}_+$ of boundary $\de \O= \G_1 \cup \G_2$ such that $0 \in \de \O \cap \G_2$, $s \in [0,1)$ and  $h \in \calC^0(\de \O \cap \G_2)$.
We let $w$ to be the ground state solution of the Hardy-Sobolev trace best constant $S_{N,s}.$ We assume that the mean curvature of the boundary $H_{\de \O}$ satisfies 
\begin{equation}\label{MeanCurvature}
\left(\frac{N-2}{2 N}+\frac{1}{N}\frac{\displaystyle\int_{\R^{N+1}_+} z_1 \left|\frac{\de w}{\de z_1}\right|^2 dz}{\displaystyle \int_{\R^{N+1}_+} z_1 |\n w|^2 dz}\right)H_{\de \O} (0) +h(0)<0.
\end{equation}
Then $\mu_{s}(\O) < S_{N,s}$ and $\mu_{s}(\O)$ is achieved by a positive function $u \in H^1(\O)$ satisfying
 \begin{align}
 \begin{cases}
\displaystyle \D u= 0& \qquad \textrm{ in } \O\vspace{2mm}\\
\displaystyle u=0 & \qquad \textrm{ on } \G_1 \vspace{2mm}\\
\displaystyle  \frac{\de u}{\de \nu}=h(x) u +d(x)^{-s} u^{q(s)-1}  & \qquad \textrm{ on } \G_2,
 \end{cases}
 \end{align}
 where $\nu$ is the unit outer normal of $\de \O$.
\end{theorem}

We mention that the study of the effect of the curvature in the Hardy-Sobolev trace inequality seems to be quite rare in the literature, see for instance the paper of the author with Fall and Minlend \cite{FMT}. While the Sobolev  ($\s=0$) inequality have been  intensively studied in the last years. Our argument of proof is based on blow up analysis (see Proposition \ref{prop:exist_smO}). The paper is organized as follows. in Section \ref{Sect1}, we recall some geometric results. In Section  \ref{Sect2}, we compare the two Hardy-Sobolev trace inequalities in order to get the existence of minimizers, see Section \ref{Sect3}. Section \ref{Sect4} is devoted to the proof of our main result.

\section{Preliminaries}\label{Sect1}
We let   $E_i$, ${i=2,\dots,N+1}$ be an  orthonormal basis of $ T_0\de\O$, the tangent plane of $\de \O$ at 0. 
 We will consider the Riemaninan manifold
 $(\de\O,\ti{g})$
where $\ti{g}$ is the Riemannian metric  induced by $\R^{N+1} $ on $\de \O$.
We  first introduce  geodesic normal coordinates in a neighborhood (in $\de\O$) of
$0$ with coordinates $ y'=(y^2,\dots,y^{N+1})\in\R^N.$ We set
$$
f(y'):=\textrm{Exp}_{0}^{\de\O}\left(\sum_{i=2}^{N+1}y^iE_i\right).
$$
 It is clear that the geodesic distance $d$ of the boundary $\de \O$ satisfies
\be\label{eq:dftiy}
d({f(\ti{y})})=|\ti{y}|.
\ee
In addition the above  choice of coordinates induces coordinate vector-fields on $\de\O$:
 $$
Y_i(y')=f_{*}(\partial_{y^i}),\quad \textrm{ for } i=2,\dots,N+1.
$$
Let $\ti{g}_{ij}=\la Y_i,Y_j\ra $, for $i,j=2,\dots,N+1$, be the component of the metric $\ti{g}$.
We have near the origin 
$$
\ti{g}_{ij}=\delta_{ij}+O(|y|^2).
$$
We denote by $ N_{\de\O}$ the  unit normal vector field along $\de\O$ interior to $\O$. Up to rotations,
 we will assume that $N_{\de\O}(0)=E_1 $.
 For any vector field $Y$ on $T\de\O$, we define $H(Y)=d N_{\de\O}[Y]$.
The mean curvature  of $\de\O$ at 0 is given by 
$$
H_{_{\de\O}}(0)=\sum_{i=2}^{N+1}\la H(E_i), E_i\ra.
$$
 Now consider a local parametrization of a neighbourhood of $0$ in $ \R^{N+1}$ defined as
$$
F(y):={f(\ti{y})}+y^{1}N_{\de\O}(f(\ti{y})),\qquad y=(y^1,\ti{y})\in B_{r_0},
$$
where $B_{r_0} $ is a small ball centred at 0.
This yields the coordinate vector-fields in $\R^{N+1} $,
\begin{eqnarray*}
Y_i(y)&:=&F_{*}(\partial_{y^i}) \qquad i=1,\dots,N+1.
 \end{eqnarray*}
Let  $g_{ij}=\la Y_i,Y_j\ra $, for $i,j=1,\dots,N+1$, be the component of the flat metric $g$.\\
We have the following (See for instance \cite{FMT})
\begin{lemma}\label{MaMetric}
For $i, j=2,...,N+1$, Taylor expansion of the metric $g$ yields
\begin{eqnarray*}
g_{ij}&=& \delta_{ij}+2\langle H(E_i),E_j\rangle y^{1}+ O(|y|^2); \\[3mm]
g_{i1}&=& 0;\\[3mm]
g_{11}&=&1.
\end{eqnarray*}
\end{lemma}
%
%
We will need the following result proved by Fall-Minlend-Thiam [\cite{FMT}, Theorem 2.1, Theorem 2.2]. Then we have
\begin{lemma}\label{Lemma111}
 Let $s\in(0,1)$.   Then, for $z=(z_1, x) \in \R_+^* \times \R^N$, $S_{N,s}$ has a positive minimizer $w\in \calD$ that satisfies
\begin{align}\label{***}
\begin{cases}
\displaystyle \D w =0&\qquad \textrm{ in } \R^{N+1}_+, \vspace{3mm}\\
\displaystyle \frac{\de w}{\de \nu}=S_{N,s} w^{q(s)-1}|x|^{-s} &\qquad \textrm{ on } \R^N, \vspace{3mm}\\
\displaystyle\int_{\R^{N}} |x|^{-s} w^{q(s)}\,dx=1&
\end{cases}
\end{align}
where $\nu$ is the outer unit normal of $\de \O.$
Moreover we have:\\
(i) \,$w=w(z)$ only depends on $z^1$ and $|x|$, and $w$ is
strictly decreasing in $|x|$.\\
\noindent
(ii) $w(z)\leq \displaystyle\frac{C}{1+|z|^{N-1}}$ for all $z\in \ov{\R^{N+1}_+},$ for some positive constant  $C$.
\end{lemma}
In \cite{FMT}, authors used the moving plane method to prove (i). Moreover (ii) is a direct consequence of the fact that the system \eqref{***} is invariant under Kelvin transformation.
\section{Comparing  $ \mu_s(\O)$ and $ S_{N,s}$}\label{Sect2}
In this section, we construct a test function for the Hardy-Sobolev trace best constant $\mu_s(\O)$ in order to compare it with $S_{N,s}$. Then we recall
$$
\mu_s(\O)=\inf_{u \in H^1(\O)\setminus \lbrace 0\rbrace} J(u),
$$ 
where the function $J$ is given by
$$
J(u)=\frac{\displaystyle\int_\O |\n u|^2 dz+\int_{\G_2} h u^2 dx}{\left(\displaystyle\int_{\G_2} d(x)^{-s} |u|^{q(s)} dx\right)^{2/q(s)}}.
$$
Let $w \in \calD$ be the positive ground state solution (positive minimizer that satisfy the corresponding Euler-Lagrange equation) given by Lemma \ref{Lemma111} normalized so that
\be\label{Normalized}
\int_{\de \R^{N+1}_+} |x|^{-s} w^{q(s)} dx=1\qquad \textrm{and} \qquad S_{N,s}=\int_{\R^{N+1}_+} |\n w|^2 dz.
\ee
We let $\e>0$. For $r_0>0$ small fix, we define 
$$
v_\e\left(F(z)\right)=\e^{\frac{1-N}{2}} w\left(\frac{z}{\e}\right), \qquad z=(y_1,x) \in \overline{B^+_{r_0}}.
$$
Let $\eta \in \calC^\infty_c\left(F(B_{r_0})\right)$ such that 
$\eta\equiv 1$ in $F(B_{r_0/2})$
and
$0\leq \eta \leq 1 $ in $\R^{N+1}_+.$ Then we define the test function by
\be
u_\e\left(F(z)\right)=\eta\left(F(z)\right) v_\e\left(F(z)\right).
\ee
We have the following expansion.
\begin{lemma}\label{ExpansionJ}
For all $N \geq 2$, we have
\begin{equation}\label{FinaleFinaleFinale1}
\begin{array}{ll}
\displaystyle J\left(u_\e\right)=\int_{\R^{N+1}_+} |\n w|^2 dz+\e H_{\de \O}(0) \int_{B^+_{r/\e}} z_1 |\n w|^2dz&\displaystyle -\frac{2}{N}\e H_{\de \O}(0) \int_{B^+_{r/\e}} z_1 |\n_x w|^2dz\\\\
&\displaystyle h(0)\e \int_{\de B^+_{r/\e}} w^2(x)dx
+O\left(\rho(\e)\right)
\end{array}
\end{equation}
where
\begin{equation}\label{Errorrr}
\begin{array}{ll}
& \displaystyle\rho(\e)=\e^2 \int_{B^+_{r/\e}} |x|^2 |\n_x w|^2 dz+\e^3 \int_{\de B^+_{r/\e}}|x|^2 w^2(x)dx+\e \int_{\de B^+_{r/\e}\setminus \de B^+_{r/2\e}} w^2(x) dx\\\\
&\displaystyle+\e^2 \int_{B^+_{r/\e}\setminus B^+_{r/2\e}} w^2(z) dz+\int_{\de \R^{N+1}_+\setminus \de B^+_{r/\e}} |x|^{-s} w^{q(s)}dx+\e^2 \int_{\de B^+_{r/\e}} |x|^{2-s} w^{q(s)} dx+\int_{\R^{N+1}_+\setminus B^+_{r/2\e}} |\n w|^2 dz.
\end{array}
\end{equation}
\end{lemma}
\begin{proof}
We let
$$
E(u_\e):=\int_\O |\n u_\e|^2 dz+\int_{\Gamma_2} h(x) u_\e^2(x) dx.
$$
Integrating by parts, we have
$$
E\left(u_\e\right)
=\int_{\O \cap F\left(B_r\right)} |\n v_\e|^2 dz
+\int_{\G_2 \cap F\left(B_r\right)} h v_\e^2 dx+F_\e
$$
where 
$$
F_\e=\int_{\O \cap F\left(B_r\right)} \left(\eta^2-1\right)|\n v_\e|^2 dz
+\int_{\G_2 \cap F\left(B_r\right)} h\left(\eta^2-1\right) v_\e^2 dx
-\int_{\O \cap F\left(B_r\right)} \left(\eta \D \eta\right) v_\e^2 dz.
$$
By a change of variable formula, we have
{\small $$
\int_{\O \cap F\left(B_r\right)} |\n v_\e|^2 dz=\sum_{\a\b=1}^{N+1}\int_{B^+_{r/\e}} g^{\a\b}\left(\e z\right) \left(\frac{\de w}{\de z_\a}\cdot \frac{\de w}{\de z_\b}\right)(z) \sqrt{|g|}(\e z) dz.
$$}
We deduce from Lemma \ref{MaMetric} that for $i,j=2,\cdots, N+1$ that
$$
g^{11}=1;\qquad\qquad  
g^{ij}(\e z)=\d_{ij}-2\e H_{ij} z_1+O\left(\e^2 |x|^2\right)\qquad\qquad \textrm{and} \qquad\qquad
g^{i1}=0.
$$
and 
{\small
\be\label{Leral0}
\sqrt{|g|}(\e z)=1+\e H_{\de \O}(0) z_1+O \left(\e^2 |x|^2\right).
\ee
}
Therefore
{\small
\begin{equation*}
\int_{\O \cap F\left(B_r\right)} |\n v_\e|^2 dz=\int_{B^+_{r/\e}} |\n w|^2 \sqrt{|g|}(\e z) dz
-\e\frac{2}{N} H_{\de \O}(0) \int_{B^+_{r/\e}}z_1 |\n_x w|^2 \sqrt{|g|}(\e z) dz+O\left(\e^2 \int_{B^+_{r/\e}} |x|^2 |\n_x w|^2 dz\right).
\end{equation*}
}
Hence
\begin{equation}\label{Leral1}
\begin{array}{ll}
\displaystyle\int_{\O \cap F\left(B_r\right)} |\n v_\e|^2 dz=&\displaystyle \int_{B^+_{r/\e}} |\n w|^2 dz+\e H_{\de \O}(0) \int_{B^+_{r/\e}} z_1 |\n w|^2dz\\\\
&\displaystyle-\frac{2}{N}\e H_{\de \O}(0) \int_{B^+_{r/\e}} z_1 |\n_x w|^2dz+O\left(\e^2 \int_{B^+_{r/\e}} |x|^2 |\n_x w|^2 dz\right).
\end{array}
\end{equation}
By change of variable formula, continuity and \eqref{Leral0}, we get that
{\small \be\label{Leral2}
\begin{array}{ll}
\displaystyle\int_{\G_2\cap F(B_r)} h v_\e^2 dx&\displaystyle= h(0) \e\int_{\de B^+_{r/\e}} w^2(x) \sqrt{|g|}(0, \e x) dx=h(0)\e \int_{\de B^+_{r/\e}} w^2(x)dx+O\left(\e^3 \int_{\de B^+_{r/\e}}|x|^2 w^2(x)dx\right).
\end{array}
\ee
}
By a change of variable formula, we have
{\small \be\label{Leral3}
F_\e=O\left(\int_{B^+_{r/\e}\setminus B^+_{r/2\e}}|\n w|^2 dz+\e \int_{\de B^+_{r/\e}\setminus \de B^+_{r/2\e}} w^2(x) dx+\e^2 \int_{B^+_{r/\e}\setminus B^+_{r/2\e}} w^2(z) dz\right)
\ee
}
By \eqref{Leral1}, \eqref{Leral2} and \eqref{Leral3} we obtain that
{\small 
\begin{equation}\label{Energy}
\begin{array}{ll}
\displaystyle E\left(u_\e\right)=\int_{B^+_{r/\e}} |\n w|^2 dz+\e H_{\de \O}(0) \int_{B^+_{r/\e}} z_1 |\n w|^2dz&\displaystyle -\frac{2}{N}\e H_{\de \O}(0) \int_{B^+_{r/\e}} z_1 |\n_x w|^2dz\\\\
&\displaystyle h(0)\e \int_{\de B^+_{r/\e}} w^2(x)dx
+O\left(\rho_1(\e)\right)
\end{array}
\end{equation}
}
where
{\small $$
\begin{array}{ll}
\displaystyle\rho_1(\e)=\e^2 \int_{B^+_{r/\e}} |x|^2 |\n_x w|^2 dz&\displaystyle+\e^3 \int_{\de B^+_{r/\e}}|x|^2 w^2(x)dx
+\int_{B^+_{r/\e}\setminus B^+_{r/2\e}}|\n w|^2 dz\\\\
&\displaystyle+\e \int_{\de B^+_{r/\e}\setminus \de B^+_{r/2\e}} w^2(x) dx+\e^2 \int_{B^+_{r/\e}\setminus B^+_{r/2\e}} w^2(z) dz.
\end{array}
$$}
We have
$$
\int_{\G_2} d(x)^{-s} |u_\e|^{q(s)} dx= \int_{\G_2 \cap F(B_r)} d(x)^{-s} |v_\e|^{q(s)} dx+O\left(\int_{\G_2 \cap F(B_r)\setminus F(B_{r/2})} d(x)^{-s} |v_\e|^{q(s)} dx\right).
$$
By a change of variable formula and \eqref{Leral0} we have
{\small
$$
\begin{array}{ll}
\displaystyle \int_{\G_2} d(x)^{-s} |u_\e|^{q(s)} dx&\displaystyle=\int_{\de B^+_{r/\e}} |x|^{-s} w^{q(s)} \sqrt{|g|}(0, \e x) dx\\\\
&\displaystyle=
\int_{\de B^+_{r/\e}} |x|^{-s} w^{q(s)}dx+O\left(\e^2 \int_{\de B^+_{r/\e}} |x|^{2-s} w^{q(s)} dx\right)\\\\
&\displaystyle
=\int_{\de \R^{N+1}_+} |x|^{-s} w^{q(s)}dx+O\left(\rho_2\left(\e\right)\right)=1+O\left(\rho_2\left(\e\right)\right),
\end{array}
$$
}
where
$$
\rho_2(\e)=\int_{\de \R^{N+1}_+\setminus \de B^+_{r/\e}} |x|^{-s} w^{q(s)}dx+\e^2 \int_{\de B^+_{r/\e}} |x|^{2-s} w^{q(s)} dx.
$$
Therefore by Taylor expansion, we get
\begin{equation}\label{Denominateur}
\left(\int_{\G_2} d(x)^{-s} |u_\e|^{q(s)} dx\right)^{2/q(s)}=1+O\left(\rho_2(\e)\right).
\end{equation}
Using \eqref{Energy} and \eqref{Denominateur}, we obatin that
\begin{equation}\label{FinaleFinaleFinale}
\begin{array}{ll}
\displaystyle J\left(u_\e\right)=\int_{\R^{N+1}_+} |\n w|^2 dz+\e H_{\de \O}(0) \int_{B^+_{r/\e}} z_1 |\n w|^2dz&\displaystyle -\frac{2}{N}\e H_{\de \O}(0) \int_{B^+_{r/\e}} z_1 |\n_x w|^2dz\\\\
&\displaystyle h(0)\e \int_{\de B^+_{r/\e}} w^2(x)dx
+O\left(\rho(\e)\right)
\end{array}
\end{equation}
where
$$
\rho(\e)=\rho_1(\e)+\rho_2(\e)+\int_{\R^{N+1}_+\setminus B^+_{r/\e}} |\n w|^2 dz.
$$
This ends the proof.
\end{proof}
We will compute the error term in the following. Then we have
\begin{lemma}\label{LemmaError}
Let $\rho(\e)$ to be the error term given by Proposition \ref{Prop1}. Then we have
$$
\rho(\e)=o(\e) \qquad \forall N \geq 3
$$
and in particular  by Proposition \ref{Prop1}, we have for all $N\geq 3$ that
$$
\begin{array}{ll}
J\left(u_\e\right)=
\displaystyle S_{N,s}+\e H_{\de \O}(0) \int_{B^+_{r/\e}} z_1 |\n w|^2dz -\frac{2}{N}\e H_{\de \O}(0) \int_{B^+_{r/\e}} z_1 |\n_x w|^2dz+ h(0)\e \int_{\de B^+_{r/\e}} w^2(x)dx
+o\left(\e\right).
\end{array}
$$
\end{lemma}
\begin{proof}
We recall that the ground state solution $w$ satisfies 
\begin{equation}\label{Estim123}
w(z) \leq \frac{C}{1+|z|^{N-1}} \qquad \textrm{in } \R^{N+1}_+.
\end{equation}
Then letting
$$
\begin{array}{ll}
\displaystyle S_1(\e):=&\displaystyle  \e^3 \int_{\de B^+_{r/\e}}|x|^2 w^2(x)dx
+\e \int_{\de B^+_{r/\e}\setminus \de B^+_{r/2\e}} w^2(x) dx+\e^2 \int_{B^+_{r/\e}\setminus B^+_{r/2\e}} w^2(z) dz\\\\
&\displaystyle+\int_{\de \R^{N+1}_+\setminus \de B^+_{r/\e}} |x|^{-s} w^{q(s)}dx+\e^2 \int_{\de B^+_{r/\e}} |x|^{2-s} w^{q(s)} dx,
\end{array}
$$
we get by a change of variable formula that
\begin{equation}\label{S1}
S_{1}(\e)=o\left(\e\right) \qquad \textrm{for } N\geq 3.
\end{equation}
We let $\vp \in \calC^{\infty}_c\left(\R^{N+1}_+\setminus B^+_{r/2}\right)$ and we set $\vp_\e(z)=\vp\left(\e z\right)$. We multiply \eqref{****} by $\vp_\e$ and we integrate by parts to get
$$
\int_{\R^{N+1}_+ \setminus B^+_{r/2\e}} \vp |\n w|^2 dz
=
\frac{1}{2} \e^2 \int_{\R^{N+1}_+\setminus B^+_{r/2\e}} w^2 \D \vp_\e dz+\int_{\de \R^{N+1}_+\setminus \de B^+_{r/2\e}} \vp_\e |x|^{-s} w^{q(s)} dx.
$$
Then using the estimation \eqref{Estim123}, we obtain
\begin{equation*}\label{Er2}
S_2(\e):=\int_{\R^{N+1}_+\setminus B^+_{r/2\e}} |\n w|^2 dz
=
O\left(\e^2\int_{\R^{N+1}_+\setminus B^+_{r/2\e}} w^2 dz+\int_{\de \R^{N+1}_+\setminus \de B^+_{r/2\e}} |x|^{-s} w^{q(s)} dx\right).
\end{equation*}
Then we get
\begin{equation}\label{S2}
\int_{\R^{N+1}_+\setminus B^+_{r/2\e}} |\n w|^2 dz=O\left(\e\right) \qquad \textrm{for } N \geq 3.
\end{equation}
To finish the estimation of the error term, we let $\psi \in \calC^\infty_c\left(B^+_r\right)$ and we define $\psi_\e(z)=\psi(\e z)$. We then multiply \eqref{****} by $\psi_\e w |x|^2$ and we integrate by parts to get
$$
\int_{B^+_{r/\e}} |x|^2 \psi_\e |\n w|^2 dz=\frac{1}{2} \int_{B^+_{r/\e}} w^2 \Delta\left(\psi_\e |x|^2\right) dz+\int_{\de B^+_{r/\e}} \psi_\e w^{2^*(s)} |x|^{2-s} dx.
$$
This implies
\begin{equation*}
S_3(\e):=\e^2 \int_{B^+_{r/\e}} |x|^2 |\n w|^2 dz=O\left(\e^2\int_{B^+_{r/\e}} w^2 dz+\e^2\int_{\de B^+_{r/\e}} w^{2^*(s)} |x|^{2-s} dx\right).
\end{equation*}
Therefore
\begin{equation}\label{S3}
S_3(\e) =o\left(\e\right) \qquad \textrm{ for }\quad N \geq 3.
\end{equation}
By \eqref{S1}, \eqref{S2} and \eqref{S3}, we finally obtain
$$
S_1(\e)+S_2(\e)+S_3(\e)=\rho(\e)=o\left(\e\right) \qquad \textrm{ for }\quad N \geq 3.
$$
This ends the proof of the Lemma.
\end{proof}

\begin{proposition}\label{Prop1}
We assume that
\begin{equation}\label{Exterieur}
\left(\frac{N-2}{2 N}+\frac{1}{N}\frac{\displaystyle\int_{\R^{N+1}_+} z_1 \left|\frac{\de w}{\de z_1}\right|^2 dz}{\displaystyle \int_{\R^{N+1}_+} z_1 |\n w|^2 dz}\right)H_{\de \O} (0)+h(0)<0.
\end{equation}
Then $\mu_s(\O)<S_{N,s}$.
\end{proposition}
\begin{proof}
We recall that the ground state solution $w$ satisfies 
\begin{align}\label{****}
\begin{cases}
\displaystyle \D w =0&\qquad \textrm{ in } \R^{N+1}_+, \vspace{3mm}\\
\displaystyle \frac{\de w}{\de \nu}=S_{N,s} w^{q(s)-1}|x|^{-s} &\qquad \textrm{ on } \R^N, \vspace{3mm}\\
\displaystyle\int_{\R^{N}} |x|^{-s} w^{q(s)}\,dx=1.&
\end{cases}
\end{align}
We multiply \eqref{****} by $z_1 \phi_\e w$ and we integrate by parts to get
$$
\int_{\R^{N+1}_+\setminus B^+_{r/\e}} z_1  \phi |\n w|^2 dz=\frac{1}{2}\int_{\R^{N+1}_+\setminus B^+_{r/\e}} w^2 \D \left(\phi_\e z_1\right) dz.
$$
This implies that for all $N \geq 3$
\begin{equation}\label{Er3}
\int_{\R^{N+1}_+\setminus B^+_{r/\e}} z_1 |\n w|^2 dz+\int_{\de\R^{N+1}_+\setminus \de B^+_{r/\e}} z_1 |\n w|^2 dz=O\left(\e^2\int_{\R^{N+1}_+\setminus B^+_{r/\e}} w^2  dz+\int_{\de \R^{N+1}_+\setminus \de B^+_{r/\e}} w^2  dz\right)=o\left(\e\right).
\end{equation}
Moreover multiplying again \eqref{****} by $z_1 w$ and integrating by parts, we get
$$
\int_{\R^{N+1}} z_1 |\n w|^2 dz=-\frac{1}{2}\int_{ \R^{N+1}_+} \frac{\de w^2}{\de z_1} dz=-\frac{1}{2}\int_{\R^{N}} \int_0^{+\infty} \frac{\de w^2}{\de z_1} dz_1 dx
$$
where we can see that
$$
\int_0^{+\infty} \frac{\de w^2}{\de z_1} dz_1=\lim_{R \to +\infty}\int_0^R \frac{\de w^2}{\de z_1} dz_1
=\lim_{R \to +\infty} w^2(R,x)-w(0,x).
$$
Since 
$$
w(z) \leq \frac{C}{1+|z|^{N-1}}
$$
we obtain
$$
\int_0^{+\infty} \frac{\de w^2}{\de z_1} dz_1=-w(0,x).
$$
Therefore 
\begin{equation}\label{Walabok}
\int_{\R^{N+1}_+} z_1 |\n w|^2 dz=\frac{1}{2} \int_{\de \R^{N+1}_+} w^2 dx<+\infty \qquad \forall N \geq 3.
\end{equation}
Hence by Lemma \ref{LemmaError}, \eqref{Er3} and \eqref{Walabok}, we finally obtain for all $N\geq 3$ that
\begin{equation}\label{Oup1}
\begin{array}{ll}
\displaystyle J(u_\e)=S_{n,s}&\displaystyle+\e \left(\frac{N-2}{N} H_{\de \O}(0)+2h(0)\right) \int_{\R^{N+1}_+} z_1 |\n w|^2 dz+\frac{2}{N} \e H_{\de \O}(0) \int_{\R^{N+1}_+} z_1 \left|\frac{\de w}{\de z_1}\right|^2 dz+o\left(\e\right).
\end{array}
\end{equation}
Since
$$
\mu_s\left(\O\right)\leq J\left(u_\e\right).
$$
we have $\mu<S_{N,s}$ provided that
$$
\left(\frac{N-2}{N} H_{\de \O}(0)+2h(0)\right) \int_{\R^{N+1}_+} z_1 |\n w|^2 dz+\frac{2}{N} H_{\de \O}(0) \int_{\R^{N+1}_+} z_1 \left|\frac{\de w}{\de z_1}\right|^2 dz<0.
$$
That is
$$
\left(\frac{N-2}{2 N}+\frac{1}{N}\frac{\displaystyle\int_{\R^{N+1}_+} z_1 \left|\frac{\de w}{\de z_1}\right|^2 dz}{\displaystyle \int_{\R^{N+1}_+} z_1 |\n w|^2 dz}\right)H_{\de \O} (0)+h(0)<0
$$
that ends the proof.
\end{proof}
\section{Existence of minimizer for $\mu_s\left(\O\right)$}\label{Sect3}
It is clear from Proposition \ref{Prop1} that the proof of Theorem \ref{Theorem1.1} should be finalized by the following two results in this section. Then we have
\begin{proposition}\label{prop:exist_smO}
Let $\O$ be a smooth bounded domain of $\R^{N+1}$ of boundary $\de \O= \Gamma_1 \cup \Gamma_2$ and such that $0\in \de \O \cap \Gamma_2$, $h\in \calC^0\left(\de \O \cap \Gamma_2\right)$ and $s \in (0,2)$. Assume that 
$\mu_s(\O)< S_{N,s}$. Then there exists a minimizer for $\mu_s(\O).$
\end{proposition}
\begin{proof}
We define $\Phi,\Psi: H^1(\O)\to \R$ by
$$
\Phi(u):=\frac{1}{2}\left(  \int_{\O}|\n u|^2dz 
+\int_{\Gamma_2} h(x) u^2dx \right)
$$
and
$$
\Psi(u)=\frac{1}{q(s)}\int_{\Gamma_2}d^{-s}(\s) |u|^{q(s)}d x.
$$
By Ekland variational principle there exits a minimizing sequence  $u_n$  for the quotient $\mu:=\mu_s\left(\O\right)$ such that
\be\label{eq:uenorm}
\int_{\Gamma_2}d^{-s}(x) |u_n|^{q(s)}dx=1,
\ee
\be
\Phi(u_n)\to \frac{1}{2} \mu_s(\O)
\ee
and
\be\label{eq:ueps-stf}
\Phi'(u_n)-\mu_s(\O)\Psi'(u_n)\to 0\quad\textrm{ in } (H^{1}(\O))',
\ee
with $(H^{1}(\O))' $ denotes the dual of  $H^{1}(\O)$.
We have that
\be\label{eq:uebndH1}
 \int_{\O}|\n u_n|^2dx  +\int_{\Gamma_2} h(x) {u_n}^2d x \leq Const.\quad \forall n\geq 1.
\ee
In particular, by coercivity,  ${u_n} \rightharpoonup u$ for some $u$ in $ H^1(\O)$.\\
\textbf{Claim:}  $u\neq0$.\\
\noindent
Assume by contradiction that $u=0$ (that is blow up occur). By continuity, \eqref{eq:uenorm} and the fact that $s\in(0,1]$, there exits a sequence $r_{n}>0$ such that
\be\label{eq:concentr}
\int_{\Gamma_2\cap B_{{r_n}}}d^{-s}(\s) |{u_n}|^{q(s)}d\s=\frac{1}{2}.
\ee
We now show that, up to a subsequence, $r_n\to0$. Indeed, by  \eqref{eq:uenorm} and \eqref{eq:concentr}
$$
\int_{\Gamma_2\setminus B_{r_n}}d^{-s}(\s) |{u_n}|^{q(s)}d\s=\frac{1}{2}.
$$
 Since $q(s)<q(0)=2^\sharp$ for $s>0$, by compactness we have
$$
  r_n^s\,C\,\leq \int_{\Gamma_2 \setminus B_{r_n}} |{u_n}|^{q(s)}d\s\leq \int_{\Gamma_2} |{u_n}|^{q(s)}d\s\to0\quad\textrm{ as } n\to \infty,
$$
for some positive constant $C$.\\
Define $ F_n(z)=\frac{1}{r_n}F(r_n z)$ for every $z\in B^+_{\frac{r_0}{r_n}} $ and put  $(g_n)_{i,j}:=\la \de_i F_n, \de_j F_n \ra=g_{ij}(r_n z)$. Clearly 
\be\label{eq:gntogEuc}
g_n\to g_{Euc}\quad C^1(K)\quad \textrm{ for every compact set } K\subset\R^{N+1},
\ee
where $g_{Euc}$ denotes the Euclidean metric. 
Let
$$
{w_n}(z)=r_n^{\frac{N-1}{2}}{u_n}(F(r_n z))\quad \forall z\in B^+_{\frac{r_0}{r_n}}.
$$
Then we get
$$
\int_{B^N_{{r_0} } }|\ti{z}|^{-s} {w_n}^{q(s)} d\tz=(1+o(1)) \int_{B^N_{{r_0}} }|\ti{z}|^{-s} {w_n}^{q(s)} \sqrt{|g_n|} d\tz
.
$$
Hence by \eqref{eq:concentr} we have
\be\label{eq:wenolmrqe}
\int_{B^N_{{r_0} } }|\ti{z}|^{-s} {w_n}^{q(s)} d\tz=\frac{1}{2} (1+c r_n) .
\ee
Let $\eta\in C^\infty_c(F(B_{r_0}))$, $\eta\equiv 1$ on $ F(B_{\frac{r_0}{2}})$ and  $\eta\equiv 0$ on $\R^{N+1} \setminus F(B_{{r_0}})$.
We define
$$
\eta_n(z)= \eta(F(r_nz))\quad\forall z\in \R^{N+1}.
$$
We have that
\be\label{eq:nwHbd}
\|\eta_n w_n \|_{\calD}\leq C \quad\forall n\in\N,
\ee
where as usual $\calD=\calD^{1,2}(\ov{\R^{N+1}}) $. Therefore
$$
\eta_n w_n\rightharpoonup w \quad\textrm{ in } \calD.
$$
We first show that $w\neq 0$. Assume by contradiction that $w\equiv 0$. Thus  
$w_n\to 0 $ in $L^p_{loc}( \R^{N+1}_+ )$ and in $L^p_{loc}( \de \R^{N+1}_+ )$  for every $1\leq p<2^\sharp$.
Let $\vp\in C^\infty_c(B_{\frac{r_0}{2}})$ be a cut-off function such that $\vp\equiv1$ on $B_{\frac{r_0}{4}} $ and $\vp\leq 1$ in $\R^{N+1}$.
Define $$\vp_n(F(y))=\vp(r_n^{-1}y).$$
We multiply \eqref{eq:ueps-stf} by $\vp_n^2 {u_n}$ (which is bounded in $H^1(\O)$)
 and integrate by parts to get
\begin{equation}\label{K1}
\begin{array}{c}
\displaystyle \int_{ \O}\n u_n\n (\vp_n^2 {u_n}) dx +\int_{\Gamma_2} h(x) u_n^2 \vp_n^2 dx=
\displaystyle  \mu_s(\O) \int_{\Gamma_2}d^{-s}(\s) | \vp_n {u_n}|^{q(s)-2}(\vp_n u_n)^2d\s+ o(1)\\
\hspace{5cm}\displaystyle \leq \mu_s(\O) \left(\int_{\Gamma_2}d^{-s}(\s) | \vp_n {u_n}|^{q(s)}d\s\right)^{\frac{2}{q(s)}}+ o(1),
\end{array}
\end{equation}
where we have used \eqref{eq:uenorm}. By compactness, we can easy see that
\begin{equation}\label{K2}
\int_{\Gamma_2} h(x) u_n^2 \vp_n^2 dx=o(1).
\end{equation}
Then  in the  coordinate system and after integration by parts, \eqref{K1} becomes
$$
\begin{array}{c}
\displaystyle\int_{ \R^{N+1}_+}|\n(\vp {w_n})|^2_{g_n}\sqrt{|g_n|} dz  \leq \displaystyle \mu_s(\O)\left( \int_{ \de\R^{N+1}_+}|\ti{z}|^{-s} | \vp {w_n}|^{q(s)}\sqrt{|g_n|} d\tz\right)^{\frac{2}{q(s)}}+o(1).
\end{array}
$$
Therefore, by \eqref{eq:gntogEuc}, for some constant $c>0$, we have
\be\label{eq:almsCont}
\begin{array}{c}
\displaystyle(1-c r_n)\int_{ \R^{N+1}_+}|\n(\vp {w_n})|^2 dz \leq \displaystyle \mu_s(\O) \left( \int_{\de \R^{N+1}_+}|\ti{z}|^{-s} | \vp {w_n}|^{q(s)} d\tz\right)^{\frac{2}{q(s)}}+o(1).
\end{array}
\ee
Hence by the Hardy-Sobolev trace  inequality  \eqref{eq:CKNtrace}, we get 
\be\label{eq:SalmSe}
\displaystyle(1-c r_n)S_{N,s}\left(\int_{ \de\R^{N+1}_+}|\ti{z}|^{-s} | \vp {w_n}|^{q(s)}d\tz\right)^{\frac{2}{q(s)}}
 \displaystyle\leq
 \mu_s(\O)\left(\int_{ \de\R^{N+1}_+}|\ti{z}|^{-s} | \vp {w_n}|^{q(s)}
 d\tz\right)^{\frac{2}{q(s)}}
+o(1).
\ee
Since  $S_{N,s}> \mu $, we conclude that
$$
o(1)=\int_{\de \R^{N+1}_+}|\ti{z}|^{-s} | \vp {w_n}|^{q(s)}d\tz
=\int_{ B^N_{{r_0}}}|\ti{z}|^{-s} |  {w_n}|^{q(s)}d\tz+o(1)
$$
because by assumption $q(s)< 2^\sharp$. This is clearly in  contradiction with \eqref{eq:wenolmrqe} thus $w\neq0$.\\
Now pick $\phi\in C^\infty_c( \R^{N+1}\setminus \{0\} )$,
and put $ \phi_n(F(y))= \phi( r_n^{-1} y)$ for every $y\in B_{r_0}$. For $n$ sufficiently large,
 $ \phi_n\in  C^\infty_c(\ov{\O})$ and it  is bounded in $H^1(\O)$.
  We multiply \eqref{eq:ueps-stf} by $\phi_n$
 and integrate by parts to get
$$
\begin{array}{c}
\displaystyle\int_{ \O}\n u_n\n \phi_n dx =
\displaystyle  \mu_s(\O) \int_{\de \O}d^{-s}(\s) | {u_n}|^{q(s)-2}u_n \phi_n d\s+ o(1).
\end{array}
$$
Hence
$$
\begin{array}{c}
\displaystyle\int_{\R^{N+1}_+  }\la\n w_n,\n \phi\ra_{g_n} \sqrt{|g_n|}dz =
\displaystyle  \mu_s(\O) \int_{\de \R^{N+1}_+}|\tz|^{-s} | {w_n}|^{q(s)-2} w_n\phi\sqrt{|g_n|} d\tz+ o(1).
\end{array}
$$
Since $\eta_n\equiv 1$ on $B_{\frac{r_0}{2r_n}}$ and the support of $\phi$ is contained in an annulus, for $n$ sufficiently large
$$
\displaystyle\int_{\R^{N+1}_+  }\la\n (\eta_n w_n),\n \phi\ra_{g_n} \sqrt{|g_n|}dz=
\displaystyle  \mu_s(\O) \int_{\de \R^{N+1}_+}|\tz|^{-s} | {\eta_nw_n}|^{q(s)-2}\eta_n w_n \phi\sqrt{|g_n|} d\tz+ o(1).
$$
Since also $g_n$ converges smoothly to the Euclidean metric on the support of $\phi$,
 by passing to the limit, we infer that, for all $ \phi\in C^\infty_c( \R^{N+1}\setminus \{0\} )$
\be\label{eq: wstfAnn}
\displaystyle\int_{\R^{N+1}_+  }\n  w\n \phi\, dz 
=  \mu_s(\O) \int_{\de \R^{N+1}_+}|\tz|^{-s} | {w}|^{q(s)-2} w \phi\, d\tz.
\ee
Notice that $ C^\infty_c( \R^{N+1}\setminus \{0\}) $ is dense in $  C^\infty_c( \R^{N+1})$
with respect to the $H^{1}( \R^{N+1})$ norm when $N\geq 2$, see e.g. \cite{Maz}.
Consequently since $w\in \calD$, it follows  that  \eqref{eq: wstfAnn} 
holds for all $ \phi\in C^\infty_c( \R^{N+1})$ by \eqref{eq:CKNtrace}. We conclude that
\begin{align*}
 \begin{cases}
\D w=0 &\quad\textrm{ in }  \R^{N+1}_+, \vspace{3mm}\\
\displaystyle-\frac{\de w}{\de z^1}=S_{N,s} |\ti{z}|^{-s}|w|^{q(s)-2} w &\quad\textrm{ on }  \de \R^{N+1}_+,\vspace{3mm}
\\
\displaystyle\int_{ \de \R^{N+1}_+}|\ti{z}|^{-s}|w|^{q(s)}\leq 1,&  \vspace{3mm},  \\
w\neq 0.
  \end{cases}
\end{align*}
Multiplying this equation by $w$ and integrating by parts, leads to $\mu_s(\O)\geq S_{N,s}$ 
by \eqref{eq:CKNtrace} which is a contradiction and thus $u=\lim u_n\neq0$
is a minimizer for $\mu_s(\O)$.\\
\end{proof}
In the following we study the existence of minimizers for the Sobolev trace inequality.
\begin{proposition}\label{prop:exist_smOs0}
Let $\O$ be a smooth bounded domain of $\R^{N+1}$ of boundary $\de \O= \Gamma_1 \cup \Gamma_2$ and such that $0\in \de \O \cap \Gamma_2$ and $h\in \calC^0\left(\de \O \cap \Gamma_2\right)$. Assume that 
$\mu_0(\O)< S_{N,0}$. Then there exists a minimizer for $\mu_0(\O).$
\end{proposition}
\begin{proof}
Recall the Sobolev trace inequality, proved by Li and Zhu in  \cite{YanYanLi}: there  exists a positive constant $C=C(\O)$ such that for all 
$u\in H^1(\O)$, we have 
\be\label{eq:almstST}
\displaystyle   S_{N,0} \left(\int_{\Gamma_2}| u|^{2^{\sharp}}  d\s\right)^{2/2^{\sharp} }
 \leq S_{N,0} \left(\int_{\de\O}| u|^{2^{\sharp}}  d\s\right)^{2/2^{\sharp} }
 \leq  \int_{\O} |\n u|^2 dx+ C \int_{\de \O}| u|^2 d\s.
\ee
Now we let $u_n$ be a minimizing sequence for $\mu$, normalized as $\|u_n\|_{L^ {2\sharp}(\Gamma_2)}=1$.
We now show that $u=\lim u_n$ is not zero. Put $\th_n:= u_n-u$ so that $ \th_n  \rightharpoonup 0$ in $H^1(\O)$  and
$\th_n\to 0$ in $L^2(\O), L^2(\de \O)$, $L^2(\Gamma_2)$.
Moreover by Brezis-Lieb Lemma \cite{BL} and recalling \eqref{eq:uenorm},   it holds that
\be\label{eq:BrLi}
1- \lim_{n\to \infty}\int_{\Gamma_2}| \th_n|^{2^{\sharp}} d\s= \int_{\Gamma_2}| u|^{2^{\sharp}} d\s.
\ee
By using \eqref{eq:almstST}, we have
\begin{align*}
 \mu_0(\O) \left(\int_{\Gamma_2}| u|^{2^{\sharp}} d\s\right)^{2/2^{\sharp} }
 &\leq  \int_{\O} |\n  u|^2dx  + \int_{\Gamma_2} h(\s)| {u}|^2d\s\\\
 &\leq  \int_{\O} |\n  u_n|^2dx+  \int_{\Gamma_2} h(\s) | {u_n}|^2d\s
 - \int_{\O} |\n  \th_n|^2dx+o(1)\\\
 & \leq  \int_{\O} |\n  u_n|^2dx+\int_{\Gamma_2} h(\s) | {u_n}|^2d\s 
-S_{N,0}  \left(\int_{\Gamma_2}| \th_n|^{2^{\sharp}} d\s\right)^{2/2^{\sharp} } +o(1)\\\
&\leq  \mu_0(\O) -  S_{N,0}  \left(\int_{\Gamma_2}| \th_n|^{2^{\sharp}} d\s\right)^{2/2^{\sharp} }+o(1).
\end{align*}
We take the limit as $n\to \infty$ and use \eqref{eq:BrLi} to get
\begin{eqnarray*}
 \mu_0(\O) \left(\int_{\Gamma_2}| u|^{2^{\sharp}} d\s\right)^{2/2^{\sharp} } \leq \mu_0(\O) -
 S_{N,0}  \left(1- \int_{\Gamma_2}| u|^{2^{\sharp}} d\s\right)^{2/2^{\sharp} }.
\end{eqnarray*}
Thanks to the concavity of the function $t\mapsto t^{2/2^{\sharp}  }$, the above implies that $$ \int_{\Gamma_2}| u|^{2^{\sharp}} d\s\geq 1$$ whenever
$\mu < S_{N,0}$.  This completes  the proof.
\end{proof}
\section{Proof of the main Result}\label{Sect4}
The proof of Theorem \ref{Theorem1.1} is a direct consequence of Proposition \ref{Prop1}, Proposition \ref{prop:exist_smO} and  Proposition \ref{prop:exist_smOs0}.

\end{document}